\documentclass[12pt,bezier]{amsart}
\usepackage{amssymb,latexsym,graphics,amsfonts}

\begin{document}

\title[Frame Existence Problem]
 {Holomorphic Hilbert Bundles and the Frame Existence Problem}

\author[M. B. Asadi]{Mohammad B. Asadi}
\address{School of Mathematics, Statistics and Computer Science,
College of Science, University of Tehran, Tehran,
 Iran, and  School of Mathematics, Institute for Research in Fundamental Sciences
 (IPM), P.O. Box: 19395-5746, Tehran, Iran}

 \email{mb.asadi@khayam.ut.ac.ir}

\author[Z. Hassanpour-Yakhdani]{Zahra Hassanpour-Yakhdani}
\address{School of Mathematics, Statistics and Computer Science, Collage of Science, University of Tehran, Tehran, Iran}

\email{z.hasanpour@ut.ac.ir}

\
\subjclass[2010]{Primary 46L08; Secondary 42C15, 46L05}%
\keywords{Hilbert C*-modules,
frames, holomorphic bundles}


\maketitle
\renewcommand\proofname{Proof}
\newtheorem{thm}{Theorem}[section]

\newtheorem{lem}[thm]{Lemma}
\newtheorem{prop}[thm]{Proposition}
\newtheorem{cor}[thm]{Corollary}

\theoremstyle{definition}
\newtheorem{dfn}[thm]{Definition}
\newtheorem{remark}[thm]{Remark}
\newtheorem{ex}[thm]{Example}
\newtheorem{question}[thm]{Question}
\numberwithin{equation}{section}
\begin{abstract}
We show that if $A=K(l^2)+ \mathbb{C}I_{l^2}$, then there exists a
Hilbert $A$-module that possess no frames.
\end{abstract}

\maketitle
\section{Introduction}
The classical frame theory has been generalized to the setting of
Hilbert C$^*$-modules
 by Frank and Larson \cite{frk}. They concluded from Kasparov's stabilization theorem
that every countably generated Hilbert C$^*$-module over a unital
C$^*$-algebra has a standard frame. They asked in \cite[Problem
8.1]{frk}, for which C$^*$-algebra $ A $, every Hilbert $A$-module
has a frame? In 2010, Li  solved this problem in the commutative
unital case, and characterized the unital commutative
C$^*$-algebra $ A $ that every Hilbert $A$-module admits a frame
as finite dimensional one \cite{Li}.

In general case, the conjecture is as follow:

\lq \lq \textit{Every Hilbert C$^*$-module over a C$^*$-algebra $
A $ admits a frame if and only if $ A $ is a dual algebra."}

In the commutative case, Li applied the Serre-Swan theorem. This
theorem states that there is an one to one correspondence between
finitely generated projective modules over a commutative
C$^*$-algebra C($ \Omega $) and complex vector bundles over $
\Omega $ \cite{swn}.

In \cite{elt}, Elliott and Kawamura shown that the vector space of
bounded uniformly continuous holomorphic sections of every
holomorphic uniformly continuous Hilbert bundle of dual Hopf type
over pure states
 of a C$^*$-algebra $A$ admits a unique structure of right Hilbert
$A$-module.

In this paper, we give a partially affirmative response to the
above conjecture. Indeed, we have applied the Elliott-Kawamura
approach and concluded the following result:

\lq \lq \textit{If $ H $ is a separable infinite dimensional
Hilbert space and  $ A = K(H) + \mathbb{C} I_H $, where $K(H)$ is
the C$^*$-algebra of compact operators on $H$,  then there exists
a Hilbert $A$-module that possess no frames."}

\section{Holomorphic Hilbert Bundle}


Let $ A $ be a C$^*$-algebra, $ \hat{A} $ the spectrum of $ A $
and $ P(A) $ be the set of pure states of $ A $. In general,
$P(A)$ is not compact, in this case we consider $ P_{0}(A) = P(A)
\cup \lbrace 0 \rbrace $. However, we set $ P_{0}(A) = P(A)$, when
$ P(A) $ is compact.

For any $ \pi : A \longrightarrow B(H_{\pi}) $ in $ \hat{A} $, and
each unite vector $ h \in H_{\pi} $, the linear functional $ f $,
given by
 $ f(\cdot) = \langle \pi (\cdot) h, h \rangle $, is a pure state.
 In this case, we use the notations $\pi=[f]$ and $f=(\pi,
 e)$, where $e=  h \otimes  h$. Also, the unitary equivalence class of $ f $ (as a set) is equal
to
\begin{center}
$R_{1}(H_{\pi}) := \lbrace e \in B(H_{\pi}) :$ e is a rank one
projection$ \rbrace $.
\end{center}
$ R_{1}(H_{\pi}) $ has a natural holomorphic manifold structure
 that is independent of the chosen representative element in each equivalence class in $ P(A) $ \cite{elt}.
Therefore, we can identify $ P(A) $ as the disjoint union of
projective spaces, i. e.,
$$ P(A) = \bigcup_{\pi \in \hat{A} } \{ \pi \} \times R_1(H_{\pi}).$$
Then $ P_{0} (A) $ has a natural holomorphic manifold structure
and it has a
 natural uniform structure determined by the seminorms arising from evaluation at the elements of $A$.

In \cite{elt}, Elliott and Kawamura introduce the concept of
 uniformly continuous holomorphic Hilbert bundle of dual Hopf type over pure states
 of a C$^*$-algebra. In fact, we set
$$ \mathcal{H} = \lbrace B(H_{\pi}, K_{\pi}) e \rbrace_{(\pi \in \hat{A} \cup \{0\} , e \in R_{1}(H_{\pi}))}, $$
where $ K_{\pi} $ is a Hilbert space, for all $ \pi \in \hat{A} $.
If $ X(\mathcal{H})$, the vector space of bounded uniformly
continuous holomorphic sections of $\mathcal{H}$, exhausting
fibres then the pair $(\mathcal{H}, X(\mathcal{H}))$ is a
uniformly continuous holomorphic Hilbert bundle of dual Hopf type.
In this case, for any $S \in X(\mathcal{H}) $ and any $ \pi \in
\hat{A}$ there exists a operator $ S_{\pi} \in B(H_{\pi}, K_{\pi})
$ such that
\begin{equation*}
S((\pi, e)) =  S_{\pi} \circ e \qquad \qquad ( e \in
R_{1}(H_{\pi}))
\end{equation*}
As shown in \cite{elt}, $ X(\mathcal{H}) $ is a Hilbert
$A$-module. In fact for any $ S, T \in X(\mathcal{H}) $, the
$A$-valued inner product is defined by $ S^{*} T  $, where
$$ S^{*}(\pi, e) = e \circ S_{\pi}^{*} \in e B(K_{\pi}, H_{\pi}), \qquad  \text{ for all} ~ (\pi, e) \in P_{0}(A).$$
Since $ (S_{\pi}^{*} \circ T_{\pi})_{\pi \in \hat{A}} \in
\prod_{\pi \in \hat{A}}(B(H_{\pi}))$ is uniformly continuous, we
can consider $S^{*} T $ belongs to $ A $, by a Brown's result
\cite{brn}.

\section{ Frame Existence Problem}

\begin{thm}
Suppose that $A$ is a C$^*$-algebra, $f_0 \in P(A)$,
$\pi_0=[f_0]$, $H_{\pi_0}$ is a separable Hilbert space and $W$ is
a countable subset of $P(A)$ such that $f_0 \in \overline{W}
\setminus W$. If there exists a uniformly continuous holomorphic
Hilbert bundle of dual Hopf type
$\mathcal{H}=(B(H_{\pi},K_{\pi})e_{\pi})_{(\pi,e_{\pi})\in
P_0(A)}$ such that for any $ \pi \in [W] $, $ K_{\pi} $ is
separable and  $ K_{\pi_{0}} $ is nonseparable, then the Hilbert $
A $-module $X(\mathcal{H})$ possess no frames.

\end{thm}
\begin{proof}
Assume that  $ \lbrace S_j \rbrace_{j \in J} $ is a frame for
$X(\mathcal{H})$. Hence, there exist positive numbers $ C, D $
such that for any section $S \in X(\mathcal{H})$ the following
inequality holds
\begin{equation*}\label{VF}
  C S^{*} S \leq \Sigma_{j\in J} S^{*} S_j S_j^{*} S \leq D
  S^{*}S.
\end{equation*}
Hence, for every  $ \pi \in \hat{A} $, $e_\pi \in R_1(H_\pi)$ and
$ S \in X(\mathcal{H})$, we have
\begin{equation*}\label{VF1}
C S^{*} S((\pi, e_{\pi})) \leq \Sigma_{j \in J} S^{*} S_j S_j^{*}
S(\pi, e_{\pi})) \leq D S^{*} S((\pi,e_{\pi})),
\end{equation*}
so,
$$ C e_{\pi} \circ S_{\pi}^{*} \circ S_{\pi} \circ e_{\pi} \leq \Sigma_{j \in J} e_{\pi}
 \circ S_{\pi}^{*} \circ S_{j\pi} \circ e_{\pi} \circ S^{*}_{j\pi} \circ S_{\pi} \circ e_{\pi} \leq D e_{\pi} \circ S_{\pi}^{*} \circ S_{\pi} \circ e_{\pi}.$$

In particular, for any nonzero element $ x_\pi \in H_{\pi} $ we
have
$$ C \Vert S_{\pi}(x_{\pi}) \Vert^{2} \leq \Sigma_{j \in J} \mid
 \langle S_{\pi} (x_{\pi}), S_{j\pi}(x_{\pi})\rangle \mid^2 \leq D \Vert S_{\pi}(x_{\pi}) \Vert^{2}. $$
Since bounded holomorphic sections exhaust fibers, so for any
$y_{\pi} \in K_{\pi}$ there exists a section $S \in
X(\mathcal{H})$ such that $S_{\pi}(x_{\pi}) = y_{\pi}$. Thus
\begin{equation}\label{VF4}
 \\C \Vert y_{\pi} \Vert^{2} \leq \Sigma_{j \in J} \mid\langle y_{\pi}, S_{j\pi}(x_{\pi}) \mid^{2} \leq D \Vert y_{\pi} \Vert^{2}.
\end{equation}
According to inequality \ref{VF4}, for all $ \pi \in \hat{A} $, $
0 \neq x_{\pi} \in H_{\pi} $, $ 0 \neq y_{\pi} \in K_{\pi} $ the
following set has to be countable
 \begin{equation*}\label{VF5}
 F_{x_{\pi}, y_{\pi}} := \lbrace j \in J : \langle y_{\pi}, S_{j\pi} (x_{\pi}) \rangle \neq 0 \rbrace.
\end{equation*}

In particular, if $\pi \in [W]$  then $K_{\pi} $ is separable and
so it has a countable orthonormal basis as $ E_{\pi}$. Hence, for
each $\pi \in [W]$, the following set has to be countable
$$F_{\pi, x_{\pi}}:= \{ j \in J : S_{j\pi} (x_{\pi}) \neq 0 \}=
\bigcup_{y_{\pi} \in E_{\pi}} \{ j \in J : \langle y_{\pi},
S_{j\pi} (x_{\pi}) \rangle \neq 0 \}.$$ Consequently, if we write
$W= \{({\pi}_n, e_n) : n \in \mathbb{N} \}$, then $F= \bigcup_{n
\in \mathbb{N}}F_{{\pi}_n, x_n}$ is a countable set, where for any
$n\in \mathbb{N}$, $x_n \in H_{{\pi}_n}$ and  $e_n= x_n \otimes
x_n$. Also, we use the notation $f_0=({\pi}_0, e_0)$, where $e_0=
x_0 \otimes x_0$ for some $x_0 \in H_{{\pi}_0}$.

For each $ j \in F$, $ Im(S_{j\pi_{0}}) $ is a separable space,
since $ H_{\pi_{0}}$ is separable.
 Then, $ K_{0} = \langle \bigcup_{j \in F} Im(S_{j\pi_{0}}) \rangle $ is
 a separable subspace of the non-separable Hilbert space $K_{\pi_{0}}$,
 hence there exists a unit element $y_{\pi_{0}} \in K_{\pi_{0}}$ that is orthogonal
  to $ K_{0} $. Then for any $ j \in F $, $ S_{j\pi_{0}}^{\ast}(y_{\pi_{0}}) = 0. $

 On the other hands, for any $ j \in J \setminus F$, we have $ S_{j\pi_{0}}(x_0) = 0
 $, since $ (\pi_{0}, e_0) \in \overline{W}$ and $S_j$ is continuous. Thus for any $ j \in J $,
  we have $\langle y_{\pi_{0}} S_{j\pi_{0}}(x_0) \rangle = 0 $.
   By  (\ref{VF4}), $ y_{\pi_{0}} $ is equal to zero, that is a contradiction.
    Therefore, the Hilbert $ A $-module $ X(\mathcal{H})$ admits no frames.
\end{proof}

\section{$K(l^2) + \mathbb{C} $}
In the following, we consider $ A = K(H) + \mathbb{C} I_H $,
 where $ H $ is a separable infinite dimensional Hilbert space. Also, let $ \{ h_{n} \}_{n \in \mathbb{N}} $ be an orthonormal
 basis for $H$ and $e_n=h_n\otimes h_n$, for all $n \in \mathbb{N}$.

We recall that, $ \hat{A} = \{ \pi_{0}, \pi_{1} \}$, where $\pi_1=
id$ and $ \pi_{0}(T + \lambda I_H) = \lambda $, for every $T \in
K(H)$ and $\lambda \in \mathbb{C}$. Thus, we can consider
$$ P(A) = (\{ \pi_{1}\} \times R_{1} (H)) \cup \{ (\pi_{0}, 1) \}.$$
Note that in this case, $ P(A) $ is a compact hausdorff space and
also $ (\pi_{0}, 1) \in \overline{W} \setminus W$, where $ W =  \{
(\pi_{1}, e_{n}) : n \in \mathbb{N} \}$.

 \begin{thm}\label{thm}
There exists a uniformly continuous holomorphic vector bundle of
dual Hopf type over $ P(A)$ satisfying the conditions of Theorem
3.1.
 \end{thm}
 \begin{proof}

Li shown that in \cite[Lemma 2.1]{Li}, there exists an uncountable
set $\mathcal{F}$ of injective maps from $\mathbb{N}$ to
$\mathbb{N}$ such that for any distinct $f, g \in \mathcal{F}$,
$f(n) \neq g(n)$ for all but finitely many $n \neq N$, and $f(n)
\neq g(m)$ for all $n \neq m$.

Let $K_{\pi_{1}}=l_2$ with the standard basis $\{ z_n \}_{n \in
\mathbb{N}}$ and $K_{\pi_{0}}$ be a non-separable Hilbert space
with an orthonormal basis $\{h_ f \}_{f\in \mathcal{F}}$ indexed
by $\mathcal{F}$. For each $f \in \mathcal{F}$ consider the
isometry $ u_{f} : H \longrightarrow l_{2}$, given by
$u_f(h_n)=z_{f(n)}$ for all $n \in \mathbb{N}$. Also, we consider
$v_{f}: \mathbb{C} \longrightarrow K_{\pi_{0}}$ by $ v_{f}
(\lambda) = \lambda h_ f$.

Now, we can define $ S_{f} : P(A) \longrightarrow (\bigcup_{e \in
R_{1}(H) } B( H, l_{2}) e ) \cup (B(\mathbb{C}, K_{{\pi}_{0}}) 1)$
by
\begin{equation*}
 S_{f}((\pi, e)) =  \{
\begin{array}{ll}
  u_{f} \circ e \ \ \ \ \pi = \pi_{1} \\
  v_{f} \circ 1  \ \ \ \ \pi= {\pi}_{0}
\end{array}
\end{equation*}
Set $V=\{\sum_{i=1}^n \lambda_i S_{f_i} : n \in \mathbb{N},
\lambda_i \in \mathbb{C}, f_i \in \mathcal{F} \}$. We claim that
the function $(\pi, e)\mapsto ||S(\pi, e)||$ is continuous on
$P(A)$ for every $S \in V$.

For this, we note that if $S=\sum_{i=1}^m \lambda_i S_{f_i} \in
V$, then there is a finite subset $J$ of $\mathbb{N}$ such that
$f_i(n)\neq f_j(n)$, for all $n \in J^c$ and $i\neq j$. Hence, if
$e=x \otimes x$, for some unite element $x \in H$, then we have
\begin{equation*}\begin{split}
||S(\pi_1, e)||^2&= ||\sum_{i=1}^m \lambda_i u_{f_i}(x)||^2
=||\sum_{i=1}^m \lambda_i (\sum_{n=1}^\infty \langle x, h_n
\rangle z_{f_i(n)})||^2
\\& =||\sum_{i=1}^m \sum_{n \in J}\lambda_i \langle x, h_n
\rangle z_{f_i(n)}||^2 + \sum_{i=1}^m|| \sum_{n \in J^c}\lambda_i
\langle x, h_n \rangle z_{f_i(n)}||^2 \\ & = ||\sum_{i=1}^m
\sum_{n \in J}\lambda_i \langle x, h_n \rangle z_{f_i(n)}||^2 +
\sum_{i=1}^m |\lambda_i|^2(1-|| \sum_{n \in J} \langle x, h_n
\rangle z_{f_i(n)}||^2).
\end{split}\end{equation*}

 Now, if a net $\{(\pi_{1},
e_{\alpha})\}_{\alpha \in I}$ is convergent to $(\pi_{1}, e)$ (or
$(\pi_{0}, 1)$) and for every $\alpha \in I$,
$e_{\alpha}=x_{\alpha} \otimes x_{\alpha}$  for some unite element
$x_{\alpha} \in H$, then $|\langle x_{\alpha}, y \rangle |
\rightarrow | \langle x, y \rangle |$ (or $|\langle x_{\alpha}, y
\rangle | \rightarrow 0$), for all $y \in H$. Consequently, for
every $f \in S$ and $y_1, \cdot\cdot\cdot, y_N \in H$, we have
$$||\sum_{n =1}^N \langle x_\alpha, y_n \rangle z_{f(n)} || (=(\sum_{n =1}^N | \langle x_\alpha, y_n \rangle
|^2)^{\frac{1}{2}})\rightarrow ||\sum_{n=1}^N \langle x, y_n
\rangle z_{f(n)}||$$ $$(\text{or} ~ ||\sum_{n=1}^N \langle
x_\alpha, y_n \rangle z_{f(n)}|| \rightarrow 0).$$ Thus,
$||S(\pi_1, e_\alpha)||\rightarrow ||S(\pi_1, e)||$ (or
$||S(\pi_1, e_\alpha)||\rightarrow ||S(\pi_0, 1)||$). This proves
the claim.

Therefore, $V$ is a linear space of bounded holomorphic sections
with uniformly continuous norm and it exhausts each fibre. Now, as
mentioned in \cite{elt}, by Zorn's lemma, we can extend it to a
linear space $ X(\mathcal{H}) $ of the bounded holomorphic
sections
  with uniformly continuous norm, maximal with this property, and exhausting each fibre.
  Clearly, $ X(\mathcal{H}) $ satisfies the conditions of
Theorem 3.1.
 \end{proof}
 The following result can be obtained from Theorems 3.1 and
 4.2.
\begin{cor}
If $A=K(l^2)+ \mathbb{C}I_{l^2}$, then there exists a Hilbert
$A$-module that possess no frames.
\end{cor}



\end{document}